\documentclass[letterpaper]{amsart}
\usepackage{amsfonts,amssymb,amscd,amsmath,color,
enumerate,url,verbatim,amstext,amsthm,lscape,longtable,epsfig}
\usepackage[portuges,francais,ngerman,english]{babel}
\usepackage[cal=esstix,calscaled=1.11,frakscaled=.97,]{mathalfa}
\usepackage[mathscr]{euscript}
\input txdtools
\usepackage{hyperref}
\theoremstyle{plain}
\newtheorem{theorem}{Theorem}[section]
\newtheorem{lemma}[theorem]{Lemma}
\newtheorem{proposition}[theorem]{Proposition}
\newtheorem{corollary}[theorem]{Corollary}
\theoremstyle{definition}
\newtheorem{definition}[theorem]{Definition}
\newtheorem{example}[theorem]{Example}
\newtheorem{conjecture}[theorem]{Conjecture}
\newtheorem{remark}[theorem]{Remark}

\newcommand{\lra}{\longrightarrow}

\newcommand{\fpt}[1]{\textsf{fpt}({#1})}
\newcommand{\lct}[1]{\textsf{lct}({#1})}
\newcommand{\fl}[1]{\lfloor{#1}\rfloor}
\newcommand{\test}[2]{\tau\left( {#1} \bullet {#2}\,\right)}
\newcommand{\la}{\lambda}
\newcommand{\<}{\prec}
\newcommand{\repe}[1]{({#1}^{\lceil\lambda {p^e} \rceil})^{[1/{p^e}]}}
\newcommand{\re}[1]{\left({#1}\right)^{[1/{p^e}]}}
\newcommand{\req}[1]{\left({#1}\right)^{[1/{q}]}}
\newcommand{\reqs}[1]{({#1})^{[1/{q}]}}
\newcommand{\repeq}[1]{\left({#1}^{\lceil\lambda {q} \rceil}\right)^{[1/{q}]}}

\newcommand{\ree}[1]{{#1\,}^{[1/{p^e}]}}
\DeclareMathOperator{\F}{F}
\newcommand{\suc}{\succ}
\newcommand{\ini}[1]{\textsf{in}_{\<}\,{#1}}

\begin{document}
 \begin{abstract}
We study higher jumping numbers and generalized test ideals associated to determinantal ideals over a  field of positive characteristic. We work in positive characteristic to give a complete characterization of both families for ideals generated by maximal minors; and make a conjecture for the general case. 
Although these invariants are understood asymptotically, as their positive characteristic grows to infinity they essentially coincide with their characteristic zero analogues, there are known examples  where the behavior in positive characteristic differs significantly from that of characteristic zero.
 This new knowledge associated to determinantal ideals strengthens the case that test ideals and multiplier ideals exhibit essentially the same description. 
  \end{abstract}
\title[F-thresholds and Generalized test ideals]{F-thresholds and Generalized test ideals of determinantal ideals of maximal minors}
\author{I.~Bonacho dos Anjos~Henriques}
\address{School of Mathematics and Statistics, University of Sheffield, \ Hicks Building, Hounsfield Road, Sheffield {S3 7RH}, United Kingdom.}
\email{i.henriques@sheffield.ac.uk}
\date{\today}
\subjclass[2000]{13A35,14M12 (primary), 14B05, 14F18 (secondary)
}
\keywords{Determinantal ideals, Generalized Test ideals, Multiplier ideals, \\$\F$-pure thresholds, log canonical thresholds.} \thanks{Research supported by EPSRC grant EP/J005436/1.}
\maketitle
\section{Introduction}\label{Section: Introduction}\label{Section: Introduction}
 Test ideals first appeared in the context of Tight Closure theory and were proved to reflect the singularities of a ring of positive characteristic.  \emph{Generalized test ideals} in algebra, as well as \emph{multiplier ideals} in birational geometry and analysis, spurred from an effort to better measure singularities. These far more subtle measurements emerged from analytic, geometric and algebraic points of view and have been the subject of a number of recent survey articles, e.g. \cite{ST},\cite{BFS}.

Motivated by a close connection to multiplier ideals in characteristic zero, N. Hara and K. Yoshida defined \emph{generalized test ideals} as their prime characteristic analogue, cf. \cite{HY}. Whereas multiplier ideals are defined geometrically, using log resolutions (see \cite{Laz}), or even analytically using integration, test ideals are defined algebraically using the Frobenius endomorphism.  Let $R$ be an $\F$-finite regular ring of positive characteristic, the \emph{generalized test ideals} of an ideal $I$ form a non-increasing, right continuous family, $\displaystyle{\{\tau(c \bullet I)\}}$, parametrized by a positive real parameter c. The points of discontinuity in this parametrization, are called \emph{F-thresholds of I}. It is known that the F-thresholds of an ideal form a discrete subset of the rational numbers, if the ring is nice enough. The issue of discreteness and rationality was first raised in \cite{MTW}, and shown to hold for every ideal in regular, $\F$-finite rings essentially of finite type over a field of positive characteristic, see \cite{BMS1}. Contributions of Hara, Katzman, Lyubeznik, Takagi, Takahashi, Schwede,Tucker and Zhang widened the class of rings over which it rationality and discreteness are known to hold (cf. \cite{ST12, ST13}). 

The first jump, i.e. the smaller threshold of an ideal $I$, plays an important role in both characteristics. In characteristic positive $p$, it is called the \emph{F-pure threshold}, denoted  ${\textsf{fpt}\,(I_{p})}$, and defined algebraically as \ $\fpt{I_{p}}={\min\{c\in\mathbb{R}_{+}\,|\ \tau(c \bullet I_{p})\ \neq R\}}.$   In characteristic zero, the first jumping number of the multiplier ideals, is called the \emph{log canonical threshold} and  denoted $\lct{I}$. 

Defining essentially the same invariant,  $\lct{I}$ and $\fpt{I_{p}}$, transpire that \emph{the smaller their values, the worse the $\F$-singularities of the ring $R/I_{p}$.} 
In \cite{HW}  Hara and  Watanabe  show that $\fpt {f_{p}}\leq \lct {f}$\  and $\displaystyle{\lim_{p\lra\infty} \fpt {f_{p}} = \lct {f}},$ where $ {f_{p}} $ is reduced from a polynomial $f$ in characteristic zero, to a polynomial ${f_{p}}$ in characteristic $p$.  Musta\c ta-Takagi-Watanabe conjectured in \cite{MTW}, that the equality $\fpt {f_{p}} = \lct {f}$, {holds for infinitely many primes $p$}.

  Relying on a log resolution of singularities, in characteristic zero, one associates to $I$  the family of \emph{multiplier ideals of $I$}, denoted $\left\{{\mathcal{J} (c \bullet {I})}\right\}_{c\, \geq 0}$. Parametrized by a positive real number $c$, \ ${\mathcal{J} (c \bullet {I})}$ form a right continuous, non-decreasing family of ideals whose first jumping number defines  the $\lct{I}$.   Just as was the case for the thresholds  $\textsf{fpt}$ and $\textsf{lct}$ (see \cite[3.4, 6.8]{HY}, \cite[\S 4]{BFS}), it is conjectured in \cite{MTW} that a parameter-wise equality of the test ideal and the (reduction to positive characteristic of the) multiplier ideal, holds in infinitely many prime characteristics. In \cite{HY} Hara and Yoshida show that this coincidence holds for any monomial (Theorem 6.10) and any toric ideal $I$ in a $\mathbb{Q}$-Gorenstein toric ring over a field of positive characteristics. Work of Hernandez on computing the F-pure threshold of binomial and diagonal hypersurfaces  \cite{Her1},\cite{Her2}, work of Miller-Singh-Varbaro \cite{MSV} and Bhatt \cite{Bh} shine some light into the complexity of the coincidence between the reductions of multiplier ideals and generalized test ideals. Recent results of Tucker and Schwede \cite{BScTu11} claim that "test ideals and multiplier ideals are morally equivalent" in the sense that they allow a unified description, in terms of \emph{alterations}, that holds in both of the characteristics, yet, better knowledge of their precise coincidence remains an open question.

This paper considers ideals generated by maximal minors of  a matrix of indeterminates, in its polynomial ring over a field of positive characteristic. We give a complete description of their generalized test ideals and F-thresholds, thus proving that the above mentioned conjectures hold for a determinantal ideal $I$  of maximal minors. It is shown that the equality \ $\tau({c} \bullet {I_{p}})=\mathcal{J}_{_{p}} (c\bullet {I})$ \  holds for all values of $c$ and all primes $p$, where $I_{p},$ and $\mathcal{J}_{_{p}} (c\bullet {I})$ denote the reductions of $I$ and $\mathcal{J}(c\bullet {I})$ to positive characteristic (cf. \cite[3.3]{HY}, \cite[3.4,4.1]{BFS}).

In Section 1, we recall the necessary definitions, and provide some context for our computations. We recall previous work of several authors (\cite{BMS1},\cite{BMS2},\cite{HY},\\\cite{MTW}) including recent work of Miller, Singh and Varbaro \cite{MSV} proving that the $\F$-pure threshold of an ideal generated by arbitrary size minors of a matrix of indeterminates in a polynomial ring over a field of positive characteristic $p$, is independent of $p$. The formula given coincides with the known formula for the log canonical threshold of such ideals (seen as ideals in a ring of characteristic zero).

In Section 2, we prove the main result of the paper Theorem \ref{Main}. It establishes the parameter-wise coincidence of the test ideal and the reduction to positive characteristic of the multiplier ideal, of a determinantal ideal of maximal minors, regardless of the value of the prime characteristic.   Motivated by this result and others in \cite{HY} and \cite{ST13}, we formulate Conjecture \ref{coMain} regarding determinantal ideals of arbitrary size minors.

\section{Summary of F-thresholds and Generalized Test ideals}\label{Section: thresholds and test ideals}

In this section we introduce the necessary definitions and results, and provide a context for our computations. Throughout we will let $R$ be a Noetherian polynomial ring over a field $k$ of prime characteristic (or even a Noetherian regular $\F$-finite ring of prime characteristic $p$). 
\subsection{  $\F$-thresholds of ideals.}
\noindent We recall the definition of $\F$-thresholds as the limit of a rational sequence, noting its existence and finiteness 
(cf.\cite[\S2,\S3]{BMS1,BMS2}).
\begin{definition}\label{fpt}
Let ${\mathfrak a}, {\mathfrak b}$ be ideals of $R$ such that \ ${{\mathfrak a}}\subseteq\operatorname{rad}({{\mathfrak b}})$\,. For every $q=p^e\geq 1$, define the \emph{$\F$- threshold of ${{\mathfrak a}}$ with respect to ${{\mathfrak b}}$} to be
 \ $ 
c^{{{\mathfrak b}}}({{\mathfrak a}})=\lim_{q\to\infty}\ \frac{\nu^{{{\mathfrak b}}}_{{{\mathfrak a}}}(q)}{q}=\\=\sup_{e\,\geq\, 0}\ \frac{\nu^{{{\mathfrak b}}}_{{{\mathfrak a}}}(q)}{q}\,
\ $
where
\  $
\nu^{{{\mathfrak b}}}_{{{\mathfrak a}}}(q)=\max \{r\geq 0 \,|\  {{\mathfrak a}}^r \nsubseteq  {{{\mathfrak b}}}^{[q]}\}.$ 
\end{definition}
\noindent If ${\mathfrak m}$ is a maximal ideal, the $\F$- threshold of ${{\mathfrak a}}$ with respect to ${\mathfrak m}$ plays a particularly important role. We refer the reader to \cite[\S2]{BMS1} and \cite[\S3]{BMS2}.
\begin{example}\label{regseq}\cite[1.3]{MTW} \quad
If ${\mathfrak b}$ is generated by a  regular sequence of length $r$ then \ $c^{{{\mathfrak b}}}({{\mathfrak b}})=\displaystyle{\lim_{q\to\infty}\ {r(q-1)}/{q}=r.}$\ \  If ${\mathfrak b}$ is a maximal ideal, then ${c^{{{\mathfrak b}}}({{\mathfrak b}})=\dim R.}$
\end{example}
 
\subsection{  The $e^{\textrm{th}}$ root ideal.}\textrm{  }
Given $e\in \mathbb{N}$ set $q=p^e$. The ring $R$ is flat, hence projective, over $R^{q}$. Thus $R^{q}$ is an $\cap$-flat $R$-module in the sense of \cite{HH1}, every ideal ${\mathfrak a}$ of $R$ satisfies $\displaystyle{\bigcap_{\tiny{\begin{array}{l}{{\mathfrak a}} \subseteq {{\mathfrak b}}^{[q]}\\ \text{ideal }{\mathfrak b}\end{array}}}\, {{\mathfrak b}}^{[q]}\,=\, {\Big(\bigcap_{\tiny{\begin{array}{l}{{\mathfrak a}} \subseteq {{\mathfrak b}}^{[q]}\\ \text{ideal }{\mathfrak b}\end{array}}}\!\!\!{{\mathfrak b}}\Big)}{}^{[q]}},$ and we make the following definition:
\begin{definition}\label{ethroot}
Fix \ $e\geq 0$ and let \ $q = p^e\,.$ \ The \emph{$\text{e}^{\text{th}}$ root ideal of an ideal ${\mathfrak a}$} in $R$, denoted $\ree{{\mathfrak a}}$ is defined to be the unique smallest ideal\, ${\mathfrak b}$, such that \,${{\mathfrak a}}\, \subseteq\, {{\mathfrak b}}^{[q]}$.
\end{definition}

\subsubsection{ {Computing $e^{\textrm{th}}$ root ideals.\ }}{\cite[2.5]{BMS1}
}\label{Prop. 2.5 BMS (2008)}\ 
Fix $e\, \geq\, 0$ and a basis  \ $\mathcal{B}_{e}$ for  $R$ over $R^{{p^{\!e}}}\!\!\!.$\ \ If  ${{\mathfrak a}=(a_{1},\dots,a_{s})}$ is an ideal whose generators are expressed in terms of $\mathcal{B}_{e}$ as \ 
\[
a_i=\sum_{\mu\in\mathcal{B}_{e}}\, {\left(g_{i\mu}\right)}^{p^e} \mu,\quad\]
$\text{for some } \ {g_{i\mu}\in R}, \text{ and } \ { i=1,\,\dots,\,s},
\ $
then \ \ ${{\mathfrak a}}^{[1/{p^{e}}]}=\left(\{\ g_{i,\mu}\,\left|\ {\mu\in\mathcal{B}_e}\,,\ 1\leq i\leq s\right.\, \}\right).$\\
\noindent If $R=K[z_{1},\dots,z_{s}]$ is a polynomial ring over a field  $K$ of characteristic $p>0$,
and ${\mathfrak a}$ has coefficients in an $F$-finite subfield $k\subseteq K$, then one may use the basis   
\[\mathcal{B}_{e}=\{\alpha \, x_{1}^{u_{1}}\dots x_{s}^{u_{s}}\,|\ \alpha\in \mathcal{B}_{k},\  0\leq u_{1},\dots, u_{s}\leq p^{e}-1\}\,\] 
where $\mathcal{B}_{k}$ forms a basis of $k$ over $k^p$, to compute the $e^{\textrm{th}}$ root  of ${{\mathfrak a}}$ in $k[z_{1},\dots,z_{s}]$. Thought of as an ideal of $R$, it coincides with the $e^{\textrm{th}}$ root  of ${{\mathfrak a}}$ in $R,\ \ree{{\mathfrak a}}$.
\begin{proposition}\label{propertieserooth}
The following hold for ideals \,${\mathfrak a},\,{\mathfrak b}$\, of \,$R$\, and integers $e,\,\ell$ in $\mathbb{N}$:
\begin{enumerate}
\item\label{1}
${\mathfrak a}\,\subseteq\, {\left({{\mathfrak a}}^{[1/p^e]}\right)}^{[p^{e}]}$
\item\label{2}
$({{\mathfrak a}\,}^{[p^{\ell}]}\,)^{[1/p^{e}]}\,=\,{{{\mathfrak a}\,}^{[p^{\ell-e}]}}$ \ \ \ for \ $\ell\,\geq e\,$. 
\item\label{3}
If \ ${\mathfrak a}\subseteq {\mathfrak b}$, \ then \ $\ree{{\mathfrak a}}\subseteq\, \ree{{\mathfrak b}}$\,.
\item\label{4}
$\re{W^{-1}\,{\mathfrak a}\,}=W^{-1}\left(\ree{{\mathfrak a}}\right)$\ \ for any multiplicative system $W$ of $R$.
\item \label{5}
If $\phi$ is a ring isomorphism,  then \, $\phi(\ree{{\mathfrak a}})\,=\,\ree{\phi({\mathfrak a})}$.
\item\label{6}
If $R=k[\,z_{1},\dots,z_{r}\,]\subseteq k[\,z_{1},\dots,z_{r},\dots,z_{s}\,]=T$, then \  $\ree{{\mathfrak a}\,T}=\ree{{\mathfrak a}}\,T$\,.
\end{enumerate}
\end{proposition}
We refer the reader to {\cite[2.13, 2.4]{BMS1}} 
  for the proofs of (1) to (5).  Part (6)  follows from \ref{Prop. 2.5 BMS (2008)}, as a generating for set for ${\mathfrak a}$ over $R$ is a generating for set for ${\mathfrak a} T$ over $T$, and a basis of $R$ over $R^{{p^{\!e}}}\!$ can be extended to one of $T$ over $T^{{p^{\!e}}}\!$, e.g.
\begin{align*}
\mathcal{B}_{e}^{R} &=\{ \alpha z_{1}^{u_{1}}\!\dots z_{r}^{u_{r}}\,|\ \ \alpha\in {k},\ 0\leq u_{1},\dots, u_{r}\leq p^{e}-1\}\\
&\subseteq
\{\alpha z_{1}^{u_{1}}\!\dots z_{s}^{u_{s}}\,|\ \ \alpha\in k,\ 0\leq u_{1},\dots,u_{s}\leq p^{e}-1\}=\!\mathcal{B}_{e}^{T}
\end{align*}

\subsection{  Generalized test ideals.}
We now recall the notion of Generalized test ideals as introduced in \cite{HY}.
\begin{definition}\label{testidealdef}
Given an ideal ${\mathfrak a}$ of $R$ and a positive real parameter $\lambda$ in $\mathbb{R}^{+}$, we define  the\emph{ generalized test ideal of ${\mathfrak a}$ at $\lambda$} to be the ideal
\[
\test{\lambda}{{\mathfrak a}}=\bigcup_{e\geq 0}\, {\repe{{\mathfrak a}}}\,=\repe{{\mathfrak a}}\ \quad\text{for} \ \ e\gg0.
\] 
\end{definition}
\begin{remark}
Note that  $\test{0}{{\mathfrak a}}=R$\,, and that $\{\test{\lambda}{{\mathfrak a}}\}_{\la>0}$ defines a non-increasing, right continuous family of ideals in the parameter $\la$. Indeed:
\begin{enumerate}
\item For all \ $ \lambda\geq {\lambda^{\prime}},\ \test{\lambda}{{\mathfrak a}}\subseteq\test{\lambda^{\prime}}{{\mathfrak a}}\,;$
\item For all \ ${\lambda}$, there exists $\,\varepsilon_{\la}>0$ such that  \ $\test{\lambda}{{\mathfrak a}}=\test{\lambda^{\prime}}{{\mathfrak a}}, \ \text{for } \ {\lambda^{\prime} \in[\lambda,\lambda+\varepsilon_{\la})}\,.$
\end{enumerate}
\end{remark}

\subsubsection{  Jumping numbers are $\F$-thresholds.}
\begin{definition}
The points of discontinuity,  or "jumps",
 $\lambda\in\mathbb{R}_{+}$ for which 
 \[\test{\lambda}{{\mathfrak a}}\neq\test{\,\lambda-\delta\,}{{\mathfrak a}}\ \ \text{ for all } \ \ \delta >0,\] are called \emph{$\F$-jumping numbers of the ideal ${\mathfrak a}$.} The smallest $\F$-jumping number of ${\mathfrak a}$, 
\ $ \fpt{{\mathfrak a}}=\min\{\la\in\mathbb{R}_{+}\,|\ \test{\la}{\mathfrak a}\neq R\},$\ 
is called the \emph{$\F$-pure threshold of the ideal ${\mathfrak a}$\,.}
\end{definition}
\begin{theorem}\cite[3.1]{BMS1}\label{discretejumps}
The set of $\F$-jumping exponents of an ideal ${\mathfrak a}$ of $R$ forms a discrete subset of the rational numbers.
\end{theorem}
\begin{theorem}\cite[3.4, 6.8]{HY},\cite[2.29, 2.30]{BMS1},\cite[2.7]{MTW}\label{threshjumps}\\
The set of jumping numbers of ${\mathfrak a}$ corresponds to the set of $\F$-thresholds \ ${\{\,
c^{J}({\mathfrak a})\,|\ \, {\mathfrak a}\subseteq \operatorname{rad}(J)
\}.}$
\begin{enumerate}
\item  If  \ ${\mathfrak a}\subseteq \operatorname{rad}(J)\,,$ then\ \ $ \test{c^{J}({\mathfrak a})}{{\mathfrak a}}\subsetneq\test{\la}{{\mathfrak a}}\,,\ \text{for all} \ \ \la<c^{J}({\mathfrak a})\,.$
\item
If $\la$ is an $\F$-jumping number for ${\mathfrak a}$,\, then \ ${\mathfrak a}\subseteq \operatorname{rad}(\test{\lambda}{{\mathfrak a}})\ \text{ and } \ {\la=c^{\test{\la}{{\mathfrak a}}}({\mathfrak a})}\,.$
\end{enumerate}
\end{theorem}

\subsubsection{Brian\c con-Skoda type theorems.}
\begin{theorem}[Skoda's Theorem]\cite[2.25]{BMS1}\label{Skoda}
Let $R$ be a Noetherian regular ring of prime characteristic $p$, and ${\mathfrak a}$ be an ideal generated by $r$ elements. \\ For all \ $\lambda \geq r\,,$ \ \ $\test{\lambda} {{\mathfrak a}}
={{\mathfrak a}}
\,\cdot\,
\test{(\lambda-1)} {{\mathfrak a}}
\subseteq {{\mathfrak a}}^{\lfloor\lambda\rfloor-r+ 1}\!.$
\end{theorem}
 \begin{proof} From \cite[2.25]{BMS1} one has $\tau\left( \lambda \bullet {{\mathfrak a}}\,\right)
={{\mathfrak a}}
\,\cdot\,
\tau\left( (\lambda-1) \bullet {{\mathfrak a}}\,\right)$ \ for all \ $\lambda \geq r.$ Letting $\la=r,$ one retrieves
 $ 
 \test{r}{{\mathfrak a}}={{\mathfrak a}}
\,\cdot\,
\tau\left( (r-1) \bullet {{\mathfrak a}}\,\right)\subseteq {\mathfrak a}\,.
$ \ 
An induction argument shows that for all  $\ell\in\mathbb{N}_{0}\,,$
\ $
\test{(r+\ell)\,}{{\mathfrak a}}
={{\mathfrak a}}^{\ell+1}\,\cdot\,\test{(r-1)}{{\mathfrak a}}
\subseteq {{\mathfrak a}}^{\ell+1}.
$
\ 
For any $\la \geq r\,,$ letting $\ell=\fl{\la}-r$ one has,
\[
\test{\la}{{\mathfrak a}}\subseteq
\test{\fl{\la}}{{\mathfrak a}}
=
\test{r+ (\fl{\la}-r)}{{\mathfrak a}}
\subseteq
 {{\mathfrak a}}^{\fl{\la}-r+1}\,.\qedhere
\]
\end{proof}
\noindent The following Lemma, on test ideals of ideals generated by indeterminates, can also be derived (for $F$-finite rings) from \cite[Cor 5.9]{HY} and Proposition \ref{propertieserooth}(6), in view of \cite[2.22]{BMS2}. 
\begin{lemma}\label{all1bys}
Let $R=k[z_{1},\dots,z_{s}]$ be a polynomial ring over a field with ${{\operatorname{char}} \,k=p}$. If \  ${\mathfrak a}=(z_{1},\!\dots,z_{r})$ \ then \ $\fpt{{\mathfrak a}}=r$ \ and 
\ \ $\test{\la}{{\mathfrak a}}\,=\,{{\mathfrak a}}^{\fl{\la}-r+1}\ \ \text{for all } \ \la\,\geq\, r-1\,.$
\end{lemma}
   \begin{proof} 
Letting \  $\la\geq r-1$ and $\ell=\fl{\la}-r+1$, we let $e\geq \, {\log}_{p}\, {({r}/{r+\ell-\la})}, \ {q=p^e}$ \  and  the monomial \ ${\mu= {(z_1\, z_2\,\dots\, z_r)}^{q-1}}\!$. Using notation set in \ref{Prop. 2.5 BMS (2008)}, one has ${\mu\in \,\mathcal{B}_{e}}$, and given ${g\in{{\mathfrak a}}^{\ell}}$, \ one has\ $ {g}^{q}\cdot \mu
\,\in\, {{{\mathfrak a}}^{(r+\ell-\frac{r}{q})\cdot q}}\subseteq  {{{\mathfrak a}}^{\la\cdot q}}$.\ \ 
 As discussed in \ref{Prop. 2.5 BMS (2008)}, ${g}\in \reqs{{{\mathfrak a}}^{\la\cdot q}},$\ \ thus\ ${\mathfrak a}^{\ell}\subseteq\reqs{{{\mathfrak a}}^{\la\cdot q}}$. Considering large enough values of $q$ yields \ ${{\mathfrak a}^{\ell}\subseteq\test{\la}{{\mathfrak a}}}.$ \ For $\la<r$, it implies that \ $R\subseteq\test{\la}{{\mathfrak a}}$,\ thus \ $\test{\la}{{\mathfrak a}}=R.$\  Skoda's Theorem \eqref{Skoda}, establishes the reverse inclusion when  $\la\geq r$, \ and we conclude that \ ${\test{\la}{{\mathfrak a}}={{\mathfrak a}}^{\ell}}\,= \, {{\mathfrak a}}^{\lfloor{\la}\rfloor -r+1}$\ and  \ $\fpt{{\mathfrak a}}=r.$
 \end{proof}
  
\section{Generalized test ideals of Determinantal ideals}
Throughout this section we let $R$ be the polynomial ring $k[x_{11},\dots,x_{mn}]$, over a field $k$ of prime characteristic $p$, and let $I=I_{_{t}}(\!X_{_{m\times n}}\!)$ be the ideal generated by $t$-minors of an ${m\times n}$ matrix of indeterminates ${X_{_{m\times n}}\!\!=\left[\,x_{i,j}\,\right]}$.  We give a complete characterization of the test ideals, when $t=m$, and of  the set of $\F$-thresholds as the set of all integers greater than or equal to the $\F$-pure threshold $\fpt{I}$.
\subsection{Elementary row operations.}\cite{BV}\label{BV}
Consider  the matrix of indeterminates 
\begin{equation}\tag{$\star$}\label{row}
\underline{Y\!}\,=\left[\,y_{i,j}\,|\ \tiny{2\,\leq\, i\,\leq\, m-1\,,\, 2\,\leq j \,\leq n-1}\,\,\right]
 \end{equation} 
and let  \ $I^{\prime}=I_{m-1}(\,\underline{Y\!}\,)$ \ be the ideal generated by its ${(t-1)}$-minors in $k[\,\underline{Y\!}\,]\,.$
   Let \[ T=k[\{\,y_{ij},\,|\ 2\leq i\leq m-1,\, 2\leq j \leq n-1\}]\,[x_{11},\cdots,x_{1n},x_{21},\cdots,x_{m1}]\,\,[\,x_{11}^{-1}\,]\]
 and \ $S=k[x_{11},\cdots,x_{mn}]\,[\,x_{11}^{-1}\,]\,.$ \ The  isomorphism $\varphi:  S\rightarrow T$  induced by:
\[\{x_{ij}\ \mapsto\  y_{ij}+x_{1j}\,x_{i1}\,x_{11}^{-1},\quad
x_{1j}\ \mapsto\ x_{1j},\quad
x_{i1}\ \mapsto\ x_{i1}\}_{\ 2\leq i\leq m-1,\  2\leq j \leq n-1}\]
has inverse,\, $\psi:\  T\rightarrow S$, \ induced by the reverse substitution:
\[\{y_{ij}\ \mapsto\  x_{ij}-x_{1j}\,x_{i1}\,x_{11}^{-1},\quad
x_{1j}\ \mapsto\ x_{1j},\quad
x_{i1}\ \mapsto\ x_{i1}\}_{\ 2\leq i\leq m-1,\  2\leq j \leq n-1}.\]
Note that $I^{(\ell)}=I^{\ell}$ and ${I^{\prime}}^{\, (\ell)}={I^{\prime}}^{\, \ell}$ for all \ $\ell\in\mathbb{N},$ as symbolic powers coincide with the usual powers for ideals generated by of maximal minors of a matrix of indeterminates. For all \ $\ell\in\mathbb{N},$ \ $I^{\ell}\,=\, I^{\ell}S\,\cap\,R$\quad\text{and}\quad$ I^{\ell}S= {\psi\left( I^{\prime} T\right)}^{\, \ell}=\psi\left({ I^{\prime} }^{\, \ell}\, T\right)\!, $ { as follows from the discussions in \cite[\S 10]{BV}.  

\subsection{The $\F$-pure threshold of determinantal ideals}
\text{}\\
We recall the following result of Miller, Singh and Varbaro yielding a formula for the $\F$-pure threshold of a  determinantal ideal \cite{MSV}:
\begin{theorem}\cite{MSV}\label{MSV}
\ Let $R=k[x_{11},\dots,x_{mn}]$ where $k$ is a field of positive characteristic, and let $I=I_{_{t}}(\!X_{_{m\times n}}\!)$ be the ideal generated by $t$-minors of the $\underline{X\!}$. The $\F$-pure threshold of $I$ is determined by:\[ {\fpt I}={\min\Big\{ \frac{(n-k)(m-k)}{(t-k)}\, \Big|\ k=0,\dots,t-1\,\Big\}}\,.\]
\end{theorem}

This formula was first used to describe the log canonical threshold of a determinantal variety in  \cite{John}, and conveys that the connection between these
invariants is stronger than the current results may suggest (see \cite[3.4]{HY}, \cite{ST13}).
\subsection{Generalized test ideals of determinantal ideals of maximal minors}\text{}\\
We proceed to state the main result of the paper, yielding a complete description of the generalized test ideals of a determinantal ideal of maximal minors.
\begin{theorem}\label{Main}
Let $R=k[x_{11},\dots,x_{mn}]$ be a polynomial ring over a field $k$ of positive characteristic,
and let $I=I_{_{m}}(\!X_{_{m\times n}}\!)$ be the ideal generated by $m$-minors of $\underline{X\!}$.
The set of \,$\F$-thresholds of $I$ is \ $\{\,n\in\mathbb{N}\,|\ \, n\geq \fpt{I}\,\}$ and for all  \ $\la\geq\fpt{I}-1\,,$ \ \ \[{\test{\la}{I}=I^{\lfloor{\la}\rfloor  - \fpt{I}+1}}.\]
 \end{theorem} The proof of Theorem \ref{Main} will use and follow that of Proposition \ref{row}.
\begin{proposition}\label{row}
Let $R=k[x_{11},\dots,x_{mn}]$ be a polynomial ring over a field of positive characteristic 
and let $I=I_{_{m}}(\!X_{_{m\times n}}\!)$ be the ideal generated by the maximal minors of a matrix of indeterminates ${X\!}_{_{\,m\times n}}$\!. \ For all \ $\la\geq n-m\,, \ {\test{\la}{I}\subseteq I^{\fl{\la}-n+m}}\,,$ where ${n-m+1=\fpt{I}}\,.$
\end{proposition}
\begin{proof} Let $\la\geq n-m$ be given.  As remarked in
Proposition \ref{propertieserooth}, the $e^{\text{th}}$ root construction behaves well with respect to the isomorphisms, localization and extension of variables.  Using notation from \ref{row}, for $e\gg 0\,,$ and $q=p^e$, one has
 \begin{align*}
\test{\la}{I}S\, & = \,\repeq{I}S
\,\stackrel{\ref{propertieserooth}(4)}{=}\,\req {I^{\lceil{\la\cdot p^e}\rceil}S}
\,\stackrel{\ref{BV}}{=}\,\req{\psi\left({{I_{m-1}(\underline{Y\!\!}\,)}}^{\lceil{\la\cdot q}\rceil}\, T\right)}
\\
&\,\stackrel{\ref{propertieserooth}(5)}{=}\,\psi\left(\req{{I^{\prime}}^{\lceil{\la\cdot q}\rceil} \, T}\right)
\,\stackrel{\ref{propertieserooth}(6)}{=}\,\psi\left(\req{{I^{\prime}}^{\lceil{\la\cdot  q}\rceil}}\, T\right)
\,\stackrel{\ref{BV}}{=}\,\psi\left(\test{\la}{I^{\prime}}\, T\right)
\,.
\end{align*}
We argue by induction, on the size $m$ of the generating minors of the ideal.  If $m=1$, Lemma \ref{all1bys} establishes that \ $\test{\la}{I}={(\underline{X\!})}^{\fl{\la}-n+1}\ \ \text{for all } \ {\la\,\geq\, n-1}.$\  If $m>1$, the induction hypothesis 
  yields \ $\test{\la}{I_{m-1}(\underline{Y\!})}
\subseteq {I_{m-1}(\underline{Y\!})}^{{\fl{\la}}-(n-1)+(m-1)}.$ \  
Thus
\begin{align*}
\test{\la}{I}S&
= \psi\left( \test{\la}{I_{m-1}(\underline{Y\!})}\, T\right)\subseteq \psi\left( {I_{m-1}(\underline{Y\!})}^{{\fl{\la}}-(n-1)+(m-1)} \, T\right)
\\
& =\psi\left( {I^{\prime}}^{{\fl{\la}}-(n-1)+(m-1)}\, T\right)=\psi\left( {I^{\prime}}^{{\fl{\la}}-n+m}\, T\right)\ {=} \  I^{{\fl{\la}-\fpt{I_{m}}+1}}\, S.
\end{align*} 
where the last equality follows from \S\ref{BV} and, the equality $\fpt{I_{m}}=n-m+1$,\  from Theorem \ref{MSV}.
As proved in \cite[2.2, 2.3]{BC}, $x_{11}$ is not a zero divisor in \ $R/ I^{{\fl{\la}-n+m}} $, thus \ $
\test{\la}{I}\ \subseteq\ ( I^{{\fl{\la}-n+m}} \,:_{R}\, x_{11})
\ =\  I^{{\fl{\la}-n+m}}$ holds, as claimed.
\end{proof}
\noindent We proceed to proof our main result:
\begin{proof}[Proof of Theorem \ref{Main}]
First note that the inclusion  $\test{\lambda}{I}\subseteq I^{\lfloor{\la}\rfloor  - \fpt{I}+1}$  follows from Proposition \ref{row} and Theorem \ref{MSV}. To establish the reverse inclusion, let $\la>n-m$ be given, fix  $e\gg0\text{ and }\, q=p^e$. \  Setting $\varepsilon=\frac{\fpt I}{q}$, one has
%%%%$\ {\varepsilon \,<\, {\fl{\la}+1-\la}\leq 1},$ 
$ \la\,<\,\fl{\la}+1- {\varepsilon} 
\,\in\, \frac{1}{q}\mathbb{N},$ \ 
{ and  } ${\tau \left(\,  (\fl{\la}+1-{\varepsilon} )\,\bullet I\,\right)=\left(I^{\, (\fl{\la}+1-{\varepsilon} ) \,\cdot\, q} \right)^{[1/{q}]}\!\!.}$  \  As   $\tau \left(\, (\fl{\la}+1-{\varepsilon} )\,\bullet I\,\right)
\ \subseteq\
\tau \left(\, (\fl{\la}+1-({\fl{\la}+1-\la}) )\,\bullet I\,\right)
\ = \ \tau (\la \bullet I)\,,$ it is enough to show that
 \ $
I^{\lfloor{\la}\rfloor  - \fpt{I}+1}
\ \subseteq\ \left(I^{\, (\fl{\la}+1-{\varepsilon} ) \,\cdot\, q} \right)^{[1/{q}]}.$

 For each\ \ $0\leq  i \leq n-m\, ,$ \ \ set 
\[\left\{
\begin{array}{l}
\delta_{i}:=\,[1\,2\,\dots\,m\,|\, (1+i)\,(2+i)\,\dots\,(m+i)]\in  I\\
\mu_{i}:=\,x_{1(1+i)}\, x_{2(2+i)}\,\dots\,x_{m(m+i)}\\
\Delta\,:=\,{\delta_0}^{{q-1}}\,{\delta_1}^{{q-1}}\,\cdots \,{\delta_{n-m}}^{{q-1}}\ \in  I^{(n-m+1)\, ({q-1})}=I^{\fpt {I}\, ({q-1}) } \\
\text{and}\\
\eta\, :=\, {\mu_0}^{ {q-1}}\,{\mu_1}^{ {q-1}}\,\cdots \,{\mu_{n-m}}^{ {q-1}}
\end{array}
\right.
\]
Consider the lexicographical term order $\<$  induced by the variable order:
\[x_{11}\,{\suc}\,x_{12}\,{\suc}\,\dots\, {\suc}\,x_{1n}\,{\suc}\,x_{21}{\suc}\,x_{22}\,{\suc}\,\dots \,{\suc}\,x_{2n}\,{\suc}\,\dots \,{\suc}\,x_{m1}\,{\suc}\,x_{m2}\,{\suc}\,\dots \,{\suc}\,x_{mn}\, ,\]
for which one has
$\quad
\mu_{0}\,=\,\ini{\,\delta_{0}}\ ,\ \mu_{1}\,=\,\ini{\,\delta_{1}}\,,\ \dots\,,\ \mu_{n-m}\,=\,\ini{\,\delta_{n-m} }\, 
\quad$\\
(here $\,\ini {\,g}$ denotes the \textit{inicial form with respect to  the term order  $\<$\,} of a polynomial $g$).
Thus
$\ \eta= {(\ini{\delta_0})}^{ {q-1}} {(\ini{\delta_1})}^{ {q-1}} \cdots {(\ini{\delta_{n-m}})}^{ {q-1}}= \ini \Delta\ \in \textsf{supp} \Delta,$  and 
${\eta= ({x_{11}\,}\, {x_{22}\,}\dots{x_{mm}\,}\, {x_{12}\,}\, {x_{23}\,}\dots{x_{m(m+1)}\,}\  \dots\  {x_{1(n-m+1)}\,}\, {x_{2(n-m+2)}\,}\dots\,{x_{mn}\,} )^{ {q-1}}
\in\mathcal{B}_{e}\,.}$ \ \ 
One can then write 
 \begin{align*}
\Delta
&=\quad {g_{\eta}}^{q} \, \cdot\,\ini{\Delta} \ \ +\sum_{\nu\in\mathcal{B}_{e}, \,\nu\<\ini{\Delta} }\, {g_{\nu}}^{q} \cdot\,\nu\quad=\quad {g_{\eta}}^{q} \, \cdot\, \eta \ \ +\sum_{\nu\in\mathcal{B}_{e}, \,\nu\neq\eta}\, {g_{\nu}}^{q} \cdot\,\nu
\end{align*}
where, by degree considerations, $\ {g_{\eta}}\ \in k. $\ \ For all \ $g$ \ in \ $I^{\lfloor{\la}\rfloor  - \fpt{I}+1}$, \ one has
 \begin{align*}
 {g}^{q}\,{\Delta}\ 
&\in\  
I^{\fpt{I}\,\cdot\,\,( {q-1})+{(\lfloor{\la}\rfloor  - \fpt{I}+1)}\,\cdot\,q }
\,=\, I^{ ({\lfloor{\la}\rfloor +1-\frac{\fpt{I}}{q}})\,\cdot\,q}
\,=\, I^{ ({\lfloor{\la}\rfloor +1-\varepsilon})\,\cdot\,q}\\
\text{ and }
\qquad{g}^{q}\,{\Delta}
&={(g\, g_{\eta})}^{q}\cdot\ \eta \ +\sum_{\nu\in\mathcal{B}, \,\nu\neq\eta}\, {(g\, g_{\nu})}^{q} \, \nu\,.
\end{align*}  
Note  \ref{Prop. 2.5 BMS (2008)},  yields 
 \ $g \in {\left( I^{ ({\lfloor{\la}\rfloor +1-\varepsilon})\,\cdot\,q} \right)}^{[1/{q}]}$ \text{ hence  }
$I^{\lfloor{\la}\rfloor  - \fpt{I}+1}\ \subseteq \ {\left( I^{ ({\lfloor{\la}\rfloor +1-\varepsilon})\,\cdot\,q} \right)}_{.}^{[1/{q}]}$
Therefore
$\ \ \test{\la }{I}=I^{\lfloor{\la}\rfloor  - \fpt{I}+1} \ \ \text{ for all }\ \ \la\geq {\fpt{I}-1}.$
\end{proof}

\subsection{Test ideals versus Multiplier ideals.}\text{}\\
This family of ideals was shown to describe the multiplier ideals of determinantal varieties over fields of characteristic zero, in Amanda Johnson's thesis  (cf. \cite{John}). 
Letting $R$ and  $I=I_{_{t}}(\!X_{_{m\times n}}\!)$ we let $\mathcal{J}({\la}\,\bullet\,{I})$ denote the multiplier ideal of $I$ at $\la$ (cf.\cite{BL}). The ideal $I$ can be viewed as ideal of $\mathbb{Z}[x_{11},\dots,x_{mn}]$, or of ${\mathbb{Z}_{p}}[x_{11},\dots,x_{mn}]$, and the same holds for $\test{\la}{I}$ and $\mathcal{J}({\la}\,\bullet\,{I})$, from the description given in \ref{Main} and \cite{John}. Indeed, their reduction modulo $p$, as described in \cite[3.3]{HY}, \cite[3.4,4.1]{BFS} or\cite[\S 4.2]{ST}, coincides with the ideals themselves, viewed as ideals in ${\mathbb{Z}_{p}}[x_{11},\dots,x_{mn}]$.
\begin{corollary}
Let $R=k[x_{11},\dots,x_{mn}]$ be a polynomial ring over a field $k$ of positive characteristic $p$,
and let $I=I_{_{t}}(\!X_{_{m\times n}}\!)$ be the ideal generated by the maximal minors of ${X\!}{_{\,m\times n}}$. \ 
For all  \ $\la\geq0\,,$ \ \ \[{\test{\la}{I}=\mathcal{J}_{p}({\la}\,\bullet\,{I})}.\]  
\end{corollary}
 Motivated by this, and other supporting results from \cite{HY} and \cite{ST13}, we make the following conjecture:
\begin{conjecture}\label{coMain}
Let $R=k[x_{11},\dots,x_{mn}]$ be a polynomial ring over a field $k$ of positive characteristic $p$,
and let $I=I_{_{t}}(\!X_{_{m\times n}}\!)$ be the ideal generated by $t$-minors of ${X\!}{_{\,m\times n}}$  for  \ ${1\leq t\leq m\leq n}$. \  For all  \ $\la\geq0\,,$ \ \ ${\test{\la}{I}=\mathcal{J}_{p}({\la}\,\bullet\,{I})}.$  
\end{conjecture}
\section*{Acknowledgment}
The author would like to thank Matteo Varbaro for the stimulating discussions, as well as the Engineering and Physical Sciences Research Council and M. Katzman for the grant support (EP/J005436/1).

\end{document}